\documentclass{amsart}
\usepackage[utf8]{inputenc}
\usepackage{amsmath}
\usepackage{amsfonts}
\usepackage{amssymb}
\usepackage{graphicx}
\usepackage{amsthm}
\usepackage{enumerate}
\usepackage{mathrsfs}
\usepackage{color}
\usepackage{hyperref}
\hypersetup{
	colorlinks=true, 
	linktoc=all,     
	linkcolor=blue,  
}
\newcommand{\A}{\mathcal{A}}
\newcommand{\B}{B_{\ell_2}}
\newcommand{\C}{\mathbb{C}}
\newcommand{\M}{\mathcal {M}}
\newcommand{\N}{\mathbb {N}}
\newcommand{\G}{\mathcal G}
\newcommand{\D}{\mathbb D}

\theoremstyle{plain}
\newtheorem{theorem}{Theorem}[section]
\newtheorem*{theorem*}{Theorem}
\newtheorem{proposition}[theorem]{\bf Proposition}

\newtheorem{lemma}[theorem]{Lemma}
\theoremstyle{definition}

\theoremstyle{remark}
\newtheorem{remark}[theorem]{\bf Remark}
\newtheorem{example}[theorem]{\bf Example}

\begin{document}

\title[Homomorphism between uniform algebras over a Hilbert space.]{A look into homomorphisms between uniform algebras over a Hilbert space.}
	
\author[V. Dimant]{Ver\'onica Dimant}

\address{Departamento de Matem\'{a}tica y Ciencias, Universidad de San
	Andr\'{e}s, Vito Dumas 284, (B1644BID) Victoria, Buenos Aires,
	Argentina and CONICET} \email{vero@udesa.edu.ar}

\author[J. Singer]{Joaqu\'{\i}n Singer}

\address{Departamento de Matem\'{a}tica, Facultad de Ciencias Exactas y Naturales, Universidad de Buenos Aires, (1428) Buenos Aires,
	Argentina and IMAS-CONICET} \email{jsinger@dm.uba.ar}
	
\thanks{Partially supported by Conicet PIP 11220130100483  and ANPCyT PICT 2015-2299 }


\begin{abstract}
We study the vector-valued spectrum $\mathcal{M}_{u,\infty}(\B,\B)$ which is the set of nonzero algebra homomorphisms from $\A_u(\B)$ (the algebra of uniformly continuous holomorphic functions on $\B$) to $\mathcal {H}^\infty(\B)$ (the algebra of bounded holomorphic functions on $\B$). This set is naturally projected onto the closed unit ball of $\mathcal {H}^\infty(\B, \ell_2)$ giving rise to an associated fibering. Extending the classical notion of cluster sets introduced by I. J. Schark (1961) to the vector-valued spectrum we define vector-valued cluster sets. The aim of the article is to look at the relationship between fibers and cluster sets obtaining results regarding the existence of analytic balls into these sets.
\end{abstract}

\maketitle

\section{Introduction}

The discussions between eight renowned mathematicians at the Conference on Analytic Functions (Institute for Advanced Studies, 1957) gave rise to a common article published in 1961 under the pseudonym I. J. Schark \cite{Schark}. The object of study was the \textsl{maximal ideal space} (or \textsl{spectrum}) of the algebra $\mathcal {H}^\infty(\D)$ of bounded holomorphic functions over the complex unit disk. The results obtained, the questions raised and the approach considered in that article have inspired the work of many researchers since then. 

The description of fibers and cluster sets of the spectrum proposed by I. J. Schark was later studied (by several authors) for algebras of bounded analytic functions on the unit ball of an \textsl{infinite dimensional} Banach space. Even if the infinite dimensional framework is linked to a more complex picture of the spectrum, the essence of I. J. Schark's point of view has continued guiding the research.
One more step in this road has appeared at considering this perspective for a \textsl{vector-valued spectrum}; and here is where our contribution fits in.  To be more precise we need to introduce some definitions and notation.

Let $X$ be a complex Banach space with open unit ball $B_X$. We denote by $\A_u(B_X)$ the Banach algebra of holomorphic functions $f:B_X \to \C$ that are uniformly continuous on $B_X$. This is a sub-algebra of $\mathcal {H}^\infty(B_X)$, the Banach algebra of all bounded holomorphic mappings on $B_X$ with the supremum norm on $B_X$. The (scalar-valued) spectrum of this uniform algebra is the set $\M_\infty(B_X)=\{\varphi: \mathcal {H}^\infty(B_X)\to\mathbb C$ nonzero algebra homomorphisms$\}$.
Since $X^*$ is contained in $\mathcal {H}^\infty(B_X)$, there is a natural projection $\pi: \M_\infty(B_X)\to\overline B_{X^{**}}$ given by $\pi(\varphi)=\varphi|_{X^*}$. Through this projection, $\M_\infty(B_X)$ has a fibered structure: for each $z\in \overline B_{X^{**}}$ the set $\pi^{-1}(z)=\{\varphi\in \M_\infty(B_X):\ \pi(\varphi)=z\}$ is called the \textit{fiber} over $z$. The spectrum  $\M_u(B_X)=\{\varphi: \A_u(B_X)\to\mathbb C$ nonzero algebra homomorphisms$\}$ is similarly projected onto $\overline B_{X^{**}}$ and the fibering over this set is consequently defined. 

Our particular interest is to study the homomorphisms from the uniform algebra $\A_u(\B)$ to $\mathcal {H}^\infty(\B)$ as a whole set by studying the vector-valued spectrum $\M_{u,\infty}(\B,\B)$ defined as
\begin{align*}
\M_{u,\infty}(\B,\B) = \{ \Phi: \A_u(\B) \to \mathcal {H}^\infty(\B) \textrm{ non-zero algebra homomorphism}\}.
\end{align*}

The uniform algebras $\A_u(B_X)$ and $\mathcal {H}^\infty(B_X)$ are typical settings in Infinite dimensional Holomorphy. Their corresponding spectrums $\M_u(B_X)$ and $\M_\infty(B_X)$ are studied, with fibers \cite{AronColeGamelin,ColeGamelinJohnson,Farmer,AronFalcoGarciaMaestre,ChoiFalcoGarciaJungMaestre} and cluster sets \cite{AronCarandoLassalleMaestre,AronCarandoGamelinLassalleMaestre,AlvesCarando} being of particular interest.
Besides the study of specific algebras and their spectrums, homomorphisms between uniform algebras have been a subject of research in both the finite \cite{GorkinMortiniSuarez,MacCluerOhnoZhao} and the infinite dimensional setting \cite{GalindoGamelinLindstrom, AronGalindoLindstrom}. Through the vector-valued spectrum, originally defined in \cite{DiGaMaSe}, we are able to study homomorphisms between uniform algebras as a whole set and analyze the structure of this set analogously as it is done for the scalar-valued spectrum. Focusing in the infinite dimensional case, the infinite dimensional polydisk $B_{c_0}$ and the euclidean ball $\B$ appear as natural domains to extend the finite-dimensional study of Banach algebras of holomorphic mappings over $\D^n$ and the euclidean ball of  $\mathbb C^n$. Having worked on the structure of the vector valued spectrum for $X = c_0$ in \cite{DimantSinger2} we now turn our attention to $\ell_2$. In contrast to $c_0$, where scalar-spectrum $\M_u(B_{c_0})$ is simple (it can be identified with $B_{\ell_\infty})$, here the spectrum $\M_u(\B)$ is more complex, which relates to the number of homomorphisms from $\A_u(\B)$ to $\mathcal {H}^\infty(\B)$. As a result, $\M_{u,\infty}(\B,\B)$ not only provides some information about homomorphisms from $\mathcal {H}^\infty(\B)$ to itself but  also it has its own interest. 

In this paper we study  fibers and  cluster sets for the vector-valued spectrum $\M_{u,\infty}(\B,\B)$ and how they are related. We present sufficient conditions for these sets to be either isolated or to contain analytic copies of \textit{balls}.





\subsection{Fibers}

As every function in $\A_u(\B)$ can be continuously extended to $\overline{B}_{\ell_2}$, each function $g \in \overline{B}_{\mathcal {H}^\infty(\B,\ell_2)}$ naturally produces a composition homomorphism $C_g \in \M_{u,\infty}(\B,\B)$ given by 
\begin{align*}
    C_g(f) = f \circ g.
\end{align*}
Conversely, as in \cite{DiGaMaSe} we can define a projection 
\begin{align*}
    \xi: \M_{u,\infty}(\B,\B) &\to \mathcal {H}^\infty(\B,\ell_2)\\
    \Phi &\mapsto [x \mapsto (\Phi(\langle \cdot, e_n\rangle)(x))_n].
\end{align*}

Since $\xi(C_g) = g$ for all $g \in \overline{B}_{\mathcal {H}^\infty(\B,\ell_2)}$, the image of this projection is $\overline{B}_{\mathcal {H}^\infty(\B,\ell_2)}$. However, there might be additional elements $\Phi \in \M_{u,\infty}(\B,\B)$ such that $\xi(\Phi) = g$. The set $\mathscr{F}(g) = \xi^{-1}(g)$ is called the fiber over $g$. Section 2 is devoted to the picture of such fibers. There we show that the fiber over any function $g \in B_{\mathcal {H}^\infty(\B.\ell_2)}$ is large, in fact it contains an analytical copy of a ball. The remaining fibers (i.e. over functions with norm 1) are more nuanced, so we present conditions under which we can still find an analytical copy of a ball  and also we describe singleton fibers.

\subsection{Cluster sets}

It was shown in \cite{KungFu} that $\mathcal {H}^\infty(\B)$ is a dual space with a predual that we denote with $\G^\infty(\B)$, keeping the notation from \cite{Mujica}. The vector-valued spectrum $\M_{u,\infty}(\B,\B)$ can then be viewed as a subset of the unit sphere of 
\begin{align*}
    \mathcal L(\A_u(\B), \mathcal {H}^\infty(\B)) = \left(A_u(\B)\widehat\otimes_\pi \mathcal G^\infty (\B)\right)^*.
\end{align*}
We endow $\M_{u,\infty}(\B,\B)$ with the the weak-star topology from $\left(A_u(\B)\widehat\otimes_\pi \mathcal G^\infty (\B)\right)^*$ that satisfies
\begin{align}
\label{weak-star}
\Phi_\alpha \to \Phi \text{ whenever } \Phi_\alpha(f)(x) \to \Phi(f)(x) \text{ for all $x \in \B$ and $f \in \A_u(\B)$.}
\end{align} It is then readily seen that $\M_{u,\infty}(\B,\B)$ is a weak-star closed (and thus compact) subset of the unit sphere of $\mathcal L(\A_u(\B), \mathcal {H}^\infty(\B)) = \left(A_u(\B)\widehat\otimes_\pi \mathcal G^\infty (\B)\right)^*$ (see \cite[Th. 11]{DiGaMaSe}
 or \cite[p. 3]{DimantSinger}). 
 
 For vector-valued cluster sets we focus on a particular case of the limits considered in \eqref{weak-star}: those obtained for a fixed function $f$ whenever $\Phi_\alpha$ are composition homomorphisms.
 
 Let $g\in \overline{B}_{\mathcal {H}^\infty(\B,\ell_2)}$, $f \in \A_u(\B)$. We define the vector-valued cluster set of $f$ over $g$ as 

\begin{align*}
    \mathcal{C\ell}_{\M_{u,\infty}(\B,\B)}(f,g) = \{ h \in \mathcal {H}^\infty(\B) \colon \exists \ (g_\alpha) \subset B_{\mathcal {H}^\infty(\B,\ell_2)}, g_\alpha \xrightarrow{w^*}g \\\text{ satisfying } C_{g_\alpha}(f)(y) \to h(y) \ \forall y \in \B\}.
\end{align*}
Here $ g_\alpha \xrightarrow{w^*}g$ denotes the weak-star convergence associated to the identification
\begin{align*}
    \mathcal {H}^\infty(\B,\ell_2) = \mathcal{L}(\mathcal{G}^\infty(\B),\ell_2)  = (\mathcal{G}^\infty(\B) \widehat\otimes_\pi \mathcal \ell_2)^*,
\end{align*}
where the fist equality was proved in \cite{Mujica}.
 These sets are the vector-valued counterpart to the (scalar-valued) cluster sets ${C\ell}_{B_X}(f,z_0)$ for $f \in \mathcal {H}^\infty(B_X)$ and $z_0 \in \overline B_X$ defined as
\begin{align*}
    {C\ell}_{B_X}(f,z_0) = \{ \lambda \in \C \colon \exists \ (z_\alpha) \subset \B, z_\alpha \xrightarrow{w^*} z_0 \text{ satisfying } f(z_\alpha) \to \lambda\}.
\end{align*}
The study of cluster sets and cluster values for the complex unit disk $\D$ was initiated  in \cite{Schark}. This work linked the set of cluster points at a boundary point of $\D$ for a function $f \in \mathcal {H}^\infty(\D)$ with the set of evaluations $\varphi(f)$ for $\varphi$ in the scalar-valued spectrum.  
This topic in the infinite dimensional setting have been recently developed in \cite{AronCarandoGamelinLassalleMaestre,AronCarandoLassalleMaestre,JohnsonOrtegaCastillo,OrtegaCastilloPrieto, AlvesCarando}.
In section 3 we analyze the vector-valued version of cluster sets with particular interest in the size and analytical structure of such sets.

\section{Fibers in $\M_{u,\infty}(\B,\B)$}
\label{FiberSection}

Recall from the introduction that $\M_u(\B)$ is the scalar-valued spectrum 
\begin{align*}
    \M_u(\B) = \{ \varphi: \A_u(\B) \to \C \text{ continuous algebra homomorphism}\} \setminus \{0\}.
\end{align*}
In this setting, the projection $\pi$ can be computed as
\begin{align}
    \pi: \M_u(\B) &\to \ell_2 \nonumber \\
    \pi(\varphi) &= \left(\varphi(\langle \cdot , e_n \rangle)\right)_n \label{pi}.
\end{align}
 A lot of progress has been made towards describing the structure of the fibers of this spectrum. It was shown in \cite{Farmer} (see also \cite{AronCarandoGamelinLassalleMaestre}) that the fibers over elements $z \in S_{\ell_2}$ are singleton sets. For fibers over elements $z \in \B$ the situation is completely different. Cole Gamelin and Johnson proved in \cite{ColeGamelinJohnson} that the fiber over $0$ contains an analytical copy of $B_{\ell_\infty}$. This result was later extended to every $z \in \B$ in \cite{ChoiFalcoGarciaJungMaestre}. With these results in mind we focus on the vector-valued spectrum $\M_{u,\infty}(\B,\B)$ to provide a description of its fibered structure over $\overline B_{\mathcal {H}^\infty(\B,\ell_2)}$.
 
When, in \cite{DimantSinger}, we tackled the study of the fibers for the vector-valued spectrum $\M_\infty(B_X,B_Y)$ (nonzero algebra homomorphisms from $\mathcal {H}^\infty (B_X)$ to $\mathcal {H}^\infty (B_Y)$), we proposed a division into what we called \textit{interior, middle} and \textit{edge fibers}. This was based on both the norm of the function over which we were projecting and whether or not the corresponding composition homomorphism could be defined. In our current setting of $\M_{u,\infty}(\B,\B)$ we noted in the introduction that composition homomorphisms $C_g$ can be defined for \textbf{every} $g \in \overline{B}_{\mathcal {H}^\infty(\B,\B)}$. It is, however, still worth to keep a distinction of \textit{interior, middle} and \textit{edge} fibers due to the different behaviour of elements belonging to each of these classes.
They are described as follows

\begin{itemize}
     \item Interior fibers: Fibers over functions $g$ with $\|g\| < 1$ (and thus  $g(\B) \subset \B$).
    \item Middle fibers: Fibers over functions $g$ with $\|g\| = 1$  and $g(\B)\subset \B$.
    \item Edge fibers: Fibers over constant functions $g$ with $|g(x)| = 1$  for all $x\in\B$.
\end{itemize}

Note that for a holomorphic function $g \in \overline{B}_{\mathcal {H}^\infty(\B,\ell_2)}$ if $g(x)\in S_{\ell_2}$ then $g(\B)\subset S_{\ell_2}$. Also,
since $\ell_2$ is strictly convex, this implies that $g$ is a constant function.

The logic behind this split is that we want to obtain new homomorphisms in the fiber over $g$ by taking limit of composition homomorphisms that are in some sense close to $C_g$ (that is, such that the projection of the limit is still $g$). Because of the way $\| \cdot \|_2$ works, moving from \textit{interior} to \textit{edge} fibers we get less room to work and thus this distinction is then reflected in the results regarding the size of the fibers for each case. As we see next, in one extreme we can inject a ball in any interior fiber while the opposite is true for edge fibers, we show that they are always singleton sets. Middle fibers, being a transition from the interior to the edge, are the most complex case. They can be isolated, they can contain an analytic copy of a ball (we exhibit conditions for both situations) but we do not know if there could be a third possibility.

\subsection{Interior Fibers}

We have shown in \cite{DimantSinger,DimantSinger2} that under certain circumstances we can construct analytical injections for the vector-valued spectrum by building upon analytical injections for the scalar valued case. With that in mind let us take a moment to recall what is known for the scalar-valued analogue of interior fibers, that is, scalar-valued fibers over elements in the interior of the unit ball $\B$. In that direction we have the following result from Cole, Gamelin and Johnson.

\begin{theorem}{\cite[Th. 6.1]{ColeGamelinJohnson}} Suppose $X$ has a normalized basis $\{x_j\}$ which is shrinking, that is, whose associated coefficient functionals $\{L_j\}$ have linear span dense in $X^*$. Suppose furthermore that there is an integer $N \geq 1$ such that
\begin{align*}
    \sum_{j=1}^\infty |L_j(x)|^N < \infty
\end{align*}
for all $x = \sum_{j=1}^\infty L_j(x)x_j$ in $X$. Then there is an analytic injection of the polydisk $B_{\ell_\infty}$ into the fiber over 0 in $\M_\infty(B_X)$.
\end{theorem}

While not in the original statement, it is readily seen that the proof of this result applies without modification to the spectrum $\M_u(\B)$. Further, we know from \cite[Ex. 1.8]{AronDimantLassalleMaestre} that via composition with the mapping
\begin{align*}
    \beta_x(y) = \frac{1}{1+\sqrt{1 - \|x\|^2}}\left\langle \frac{x - y}{1 - \langle y, x \rangle},x\right\rangle x + \sqrt{1 - \|x\|^2}\frac{x - y}{1 - \langle y, x \rangle}
\end{align*}
we can obtain an isometric isomorphism (in the metric given by being a subset of $\A_u(\B)^*$) between the fibers $\pi^{-1}(x)$ and $\pi^{-1}(0)$ for all $x \in \B$. As a consequence, the fiber over each $x \in \B$ contains an analytic copy of $B_{\ell_\infty}$. However, we can not take advantage of this construction for the vector-valued case. The main reason being that the mappings $\beta_x$ are not ``smooth'' on $x$. Still, we can manage to get analytic injections  in the vector-valued case. For that, we first prove a new embedding result for the scalar-valued spectrum in the spirit of \cite[Th. 6.7]{ColeGamelinJohnson}. Then, we make use of this to show  how the defining property of the interior fibers ($\|g\| < 1$) ensures the injection of an analytic copy of the ball $B_{\mathcal {H}^\infty(B_{\ell_2},\ell_2)}$.

As usual, for the scalar-valued spectrum $\M_u(\B)$ we denote by $\delta_z$ the evaluation mapping given by $\delta_z(f)=f(z)$ for $f\in\A_u(\B)$ and $z\in\overline{B}_{\ell_2}$.

\begin{lemma}\label{iny-escalar}
For any $0<r<1$ there exists an analytic injection $\Phi_r:r\B\times \B\to \M_u(\B)$ such that the image of $\Phi_r(w,\cdot)$ is contained in the fiber $\pi^{-1}(w)$.
\end{lemma}

\begin{proof}
For each $k \in \N$ we define $s_{k} :r\B\times \B\to \B $ as
\begin{align*}
    s_{k} (w,z) = \sum_{j=1}^k \langle w, e_j \rangle e_j +  (1-r)\left(\sum_{j=1}^\infty\langle z,e_j\rangle  e_{k+j}\right)
\end{align*}
and $\Phi_k:r\B\times \B\to \M_u(\B)$ as $\Phi_k(w,z)=\delta_{s_k(w,z)}$.

We now fix a free ultrafilter $\mathscr{U}$ containing the sets $\{k: k\geq k_0\}$ and define  the mapping $\Phi_r:r\B\times \B\to \M_u(\B)$ where, for $(w,z)\in r\B\times \B$, $\Phi_r(w,z)$ is the weak star limit along $\mathscr{U}$ of $(\Phi_k(w,z))_k$. Note that the limit along the ultrafilter exists because $\M_u(\B)$ is a weak-star compact subset of $(A_u(\B))^*.$ Then, we have, for $(w,z)\in r\B\times \B$ and $f\in\A_u(\B)$,
\begin{align*}
    \Phi_r(w,z)(f) = \lim_{\mathscr{U}} \Phi_k(w,z)(f).
\end{align*}

We now move on to show that $\Phi_r:r\B\times \B\to \M_u(\B)$ is an analytic injection, mapping $\{w\}\times\B$ into the fiber over $w$.

To see that the image of $\Phi_r(w,\cdot)$ is contained in $\pi^{-1}(w)$ we fix $z\in\B$, $n \in \N$. We then have for every $k \geq n$ the equality
\begin{align*}
    \Phi_k(w,z)(\langle \cdot, e_n \rangle) = w_n.
\end{align*}
This implies that $\Phi_r(w,z)(\langle \cdot, e_n \rangle) = w_n$ for all $n$, concluding that $\Phi_r(w,z)\in \pi^{-1}(w)$.

Let us now see that $\Phi_r$ is analytic. For such purpose, we fix $f \in \A_u(\B)$ and we have to prove that the mapping $\widehat f\circ\Phi_r:r\B\times \B\to\C$ given by $[(w,z)\mapsto \Phi_r(w,z)(f)]$ is holomorphic.

Since the sequence $(\widehat f\circ \Phi_k)_k$ is contained in the  weak-star compact set $\|f\|\overline{B}_{\mathcal {H}^\infty(r\B\times\B)}$ it should have a  weak-star limit along the ultrafilter $\mathscr{U}$. This limit  should coincide with $\widehat f\circ\Phi_r$, meaning that this mapping belongs to $\|f\|\overline{B}_{\mathcal {H}^\infty(r\B\times\B)}$ and so $\widehat f\circ\Phi_r$ is analytic.


To verify the injectivity of $\Phi_r$, note first that when $w\not= w'$ we have $\Phi_r(w,z)\not= \Phi_r(w',z)$  since they belong to different fibers. Now, 
inspired by \cite[Th. 3.1]{ChoiFalcoGarciaJungMaestre}, we consider for each $\lambda \in \overline{\D}$ the test functions $f_\lambda(x) = \sum_{j=1}^\infty \lambda^j x_j^2$, so that
\begin{align*}
    \Phi_{k}(w,z)(f_\lambda)= \sum_{j=1}^{k} \lambda^j \langle w, e_j \rangle^2 + \lambda^{k} \sum_{j=1}^\infty \lambda^j((1-r)\langle z, e_j \rangle)^2.
\end{align*}
Let $\chi(\lambda) = \lim_{\mathscr{U}} \lambda^{k}$ which exists for all $\lambda \in \overline \D$. We can then compute 
\begin{align}
\label{Composition}
    \Phi_r(w,z)(f_\lambda)=  \sum_{j=1}^\infty \lambda^j \langle w, e_j \rangle^2 + \chi(\lambda)\sum_{j=1}^\infty \lambda^j((1-r)\langle z, e_j \rangle)^2.
\end{align}

As a consequence, for $\lambda \in \partial \D$ we get that   $\Phi_r(w,z)(f_\lambda) = \Phi_r(w,z')(f_\lambda)$ if and only if
\begin{align*}
    \sum_{j=1}^\infty \lambda^j ((1-r)\langle z, e_j \rangle)^2 = \sum_{j=1}^\infty \lambda^j ((1-r)\langle z', e_j \rangle)^2.
\end{align*}
Since $F(\lambda) = \sum_{j=1}^\infty \lambda^j ((1-r)\langle z, e_j \rangle)^2$ is a holomorphic function in $\D$, continuous in $\overline{\D}$ and coincides with $F'(\lambda) = \sum_{j=1}^\infty \lambda^j ((1-r)\langle z', e_j \rangle)^2$ over $\partial \D$, the Taylor expansion for both functions must coincide and thus $(\langle z, e_j \rangle)^2 = (\langle z', e_j \rangle)^2$ for all $j \in \N$. Repeating the argument with $\sum_{j=1}^\infty \lambda^j x_j^3$ yields the equality $z = z'$, showing that $\Phi_r$ is in fact an analytic injection. 
\end{proof}

\begin{theorem}
\label{inyVectorial}
For every $g \in B_{\mathcal {H}^\infty(B_{\ell_2,\ell_2})}$ there is an analytic injection 
$$\Psi_g: B_{\mathcal {H}^\infty(B_{\ell_2},\ell_2)} \to \mathscr{F}(g).$$
\end{theorem}

\begin{proof}
Given $g \in B_{\mathcal {H}^\infty(\B,\ell_2)}$ there exists $r >0$ such that $\|g\| < r < 1$. Considering $\Phi_r$ the mapping constructed in the previous lemma, we define  $\Psi_g: B_{\mathcal {H}^\infty(B_{\ell_2},\ell_2)} \to \mathscr{F}(g)$ by
$$
\Psi_g(h)(f)(x)=\Phi_r(g(x), h(x))(f)
$$
for each  $h \in B_{\mathcal {H}^\infty(\B,\ell_2)}$, $f\in\A_u(\B)$, $x\in\B$.

To see that this is well defined we  need to check that $\Psi_g(h)$ is an algebra homomorphism from $\A_u(\B)$ to $\mathcal {H}^\infty(\B)$ and that $\xi(\Psi_g(h))=g$. Note that from 
Lemma \ref{iny-escalar} we can derive that the mapping from $\B$ to $\M_u(\B)$ given by $[x\mapsto \Phi_r(g(x), h(x))]$ is analytic. As a consequence we have that $\Psi_g(h)(f)$ is holomorphic. Also, it is clear that is bounded resulting that $\Psi_g(h)(f)\in\mathcal {H}^\infty(\B)$. The fact that $\Psi_g(h)$ is an algebra homomorphism follows easily from Lemma \ref{iny-escalar}. Finally, we appeal again to Lemma \ref{iny-escalar} to produce the identity
$$
\xi\left(\Psi_g(h)\right)(x)= \pi \left(\Phi_r(g(x), h(x))\right)=g(x)
$$ which shows that $\Psi_g(h)\in \mathscr{F}(g)$.

The injectivity of $\Psi_g$ is also an easy consequence of the injectivity of $\Phi_r$. We finish by showing that $\Psi_g$ is analytic (i. e. each mapping $[h\mapsto \Psi_g(h)(f)(x)]$ from $B_{\mathcal {H}^\infty(\B,\ell_2)}$ to $\C$ is analytic). This is obtained through the following composition of analytic mappings (the second one due to Lemma \ref{iny-escalar}):
$$
h\mapsto (g(x), h(x))\mapsto \Phi_r(g(x), h(x))(f) = \Psi_g(h)(f)(x).
$$

\end{proof}

\begin{remark}
Note that $\Psi_g(0)=C_g$. For any  $h\not\equiv 0$ the injectivity of $\Psi_g$ along with the fact that $\xi(\Psi_g(h)) = g$ tell us that $\Psi_g(h)$ could not be a composition homomorphism. 
\end{remark}

\subsection{Edge Fibers} As we mentioned at the beginning of the section, \textit{edge fibers} are the polar opposite to \textit{interior fibers}. By relating the behaviour of the scalar-valued spectrum, particularly fibers over $S_{\ell_2}$, to that of edge fibers of the vector-valued spectrum we can show that the fiber over any constant function $g \in S_{\mathcal {H}^\infty(\B,\ell_2)}$ consists only of the corresponding composition homomorphism $C_g$. With that goal in mind we begin with the following remark giving another description of composition homomorphisms (cf \cite[Rmk. 3.1]{DimantSinger}).

\begin{remark}
\label{rmkCg}
Let $g \in \overline{B}_{\mathcal {H}^\infty(\B,\ell_2)}$ and $\Psi \in \M_{u,\infty}(B_{\ell_2}, B_{\ell_2})$. Then $\Psi = C_g$ if and only if $\delta_x \circ \Psi = [f \mapsto \Psi(f)(x)]$ coincides with $\delta_{g(x)}$ for all $x \in \B$.
\end{remark}

To complete our picture for the \textit{edge fibers} we need the following well-known result regarding fibers for the scalar spectrum $\M_u(\B)$.

\begin{lemma} \cite{Farmer}
\label{deltaAu}
For every $z \in S_{\ell_2}$ the scalar-valued fiber $\pi^{-1}(z) \subset \M_u(\B)$ consists only of the evaluation homomorphism $\delta_z$.
\end{lemma}

Now we have all the necessary ingredients to prove our main result regarding \textit{edge fibers}.

\begin{proposition}
     Let $g \in S_{\mathcal {H}^\infty(\B,\ell_2)}$ be a constant function. Then the fiber $\mathscr{F}(g)$ consists only of the composition homomorphism $C_g$.
\end{proposition}

\begin{proof}
 Let  $\Psi \in \M_{u,\infty}(\B,\B)$ such that $\xi(\Psi) = g$. For all $x \in \B$ we consider the mapping $\delta_x \circ \Psi \in \M_u(\B)$. This mapping verifies
 \begin{align*}
     \pi(\delta_x \circ \Psi) &= (\delta_x \circ \Psi ( \langle \cdot , e_n \rangle ))_n\\
     &= (\Psi ( \langle \cdot , e_n \rangle ))_n (x) = \xi(\Phi)(x) = g(x).
 \end{align*}
By Lemma \ref{deltaAu}, the previous equality implies that $\delta_x \circ \Psi = \delta_{g(x)}$ for every $x \in \B$. Recalling Remark~\ref{rmkCg} we conclude that $\Psi = C_g$.
\end{proof}

\subsection{Middle Fibers}

Middle fibers are somewhat a transition case as we can find examples for both singleton middle fibers and middle fibers containing an analytic copy of a ball. To give some context to our results we present some guiding examples.

Let $Id:\B \to \B$ be the identity function. Clearly the fiber over $Id$ is a middle fiber. However, the identity is singular in many ways. One of them being that $Id$ is continuously extendable to $\overline{B}_{\ell_2}$ and the extension verifies $Id(S_{\ell_2}) \subset S_{\ell_2}$. Recall that Lemma $\ref{deltaAu}$ shows that every $z \in S_{\ell_2}$ has a singleton fiber for the scalar-valued spectrum, so the condition $Id(S_{\ell_2}) \subset S_{\ell_2}$ turns out to be quite restrictive. We will see that the only homomorphism in $\M_{u,\infty}(\B,\B)$ lying in the fiber over $Id$ is the corresponding composition mapping (which of course is the identity mapping).

Before we go into further detail regarding the fiber over $Id$ we need to introduce some additional notation.

We denote with $\M_\infty(\B)$ the scalar-valued spectrum 
\begin{align*}
    \M_\infty(\B) = \{ \varphi: \mathcal {H}^\infty(\B) \to \C \text{ continuous algebra homomorphism}\} \setminus \{0\}.
\end{align*}
As it happens to $\M_u(\B)$, this spectrum also has an associated projection
 \begin{align}
    \pi: \M_\infty(\B) &\to \ell_2 \nonumber\\
    \pi(\varphi) &= \left(\varphi(\langle \cdot , e_n \rangle)\right)_n.
    \label{pi_infty}
\end{align}
To prevent confusion, we will write $\pi_\infty$ to denote the projection defined in \eqref{pi_infty} and $\pi_u$ to denote the projection defined in \eqref{pi}, whenever both are involved in the same computation.

Also, we recall the following result (which was used to derive Lemma \ref{deltaAu}).

\begin{lemma}{\cite[Lem. 4.4]{Farmer}}
\label{farmer}
Let $B$ be the ball of a uniformly convex Banach space. Then $f \in \mathcal {H}^\infty(B)$ is continuously extendable to a point $x$ of the unit sphere if and only if $f$ is constant on the fiber $\pi_\infty^{-1}(x)$.
\end{lemma}

We can now focus on the previously mentioned example.

\begin{example}
 Let $\Phi \in \M_{u,\infty}(\B,\B)$ such that $\xi(\Phi) = Id$ and let $C^*_\Phi$ denote the composition mapping
 \begin{align*}
     C^*_\Phi:\M_\infty(\B) &\to \M_u(\B)\\
     C^*_\Phi(\varphi) &= \varphi \circ \Phi.
 \end{align*}
 Since for every $n \in \N$ we have $C^*_\Phi(\varphi) (\langle \cdot , e_n \rangle) = \varphi(\Phi(\langle \cdot , e_n \rangle)) = \varphi (\langle \cdot , e_n \rangle)$ it follows that $\pi_u(C^*_\Phi(\varphi)) = \pi_\infty(\varphi)$. As a consequence, for every $z \in S_{\ell_2}$ and $\varphi \in \pi_\infty^{-1}(z)$ the composition homomorphism $C^*_\Phi$ maps $\varphi$ to the only homomorphism in $\pi_u^{-1}(z)$, namely $\delta_z$. 
Now for every $f \in \A_u(\B)$ we have that $\Phi(f) \in \mathcal {H}^\infty(\B)$ is constant on the fiber over any $z \in S_{\ell_2}$ (taking the value $f(z)$). By Lemma \ref{farmer}, $\Phi(f)$ has a continuous extension to $S_{\ell_2}$ where it coincides with $f$. By the maximum modulus principle it follows that $\Phi(f) = f$, showing that the identity function has a singleton fiber.
\end{example}

Now, isolating some of the core elements of the previous example we obtain the following  statement.

\begin{lemma} \label{Singleton middle fiber}
Let $g \in \mathcal {H}^\infty(\B,\ell_2)$, continuously extendable to $S_{\ell_2}$ such that $g(S_{\ell_2}) \subset S_{\ell_2}$. Then the fiber $\mathscr{F}(g)$ coincides with $\{\Phi_g\}.$
\end{lemma}

\begin{proof} 
 Let $z \in S_{\ell_2}$ and take $\varphi \in \pi_\infty^{-1}(z)$. For $\Phi \in \mathscr{F}(g)$ we have that
 \begin{align*}
    C_\Phi^*(\varphi)(\langle \cdot, e_n \rangle) = \varphi(\Phi(\langle \cdot, e_n \rangle)) = \varphi(\langle g, e_n\rangle).  
 \end{align*}
 Since $\varphi$ lies in the fiber over $z \in S_{\ell_2}$ and $g$ is continuously extendable to the unit sphere, by Lemma \ref{farmer} we conclude that  $\varphi(\langle g, e_n\rangle) = \langle g , e_n \rangle (z)$. It follows that $\pi_u(C_\Phi^*(\varphi)) = g(z)\in S_{\ell_2}$. We can now compute for $f \in \A_u(\B)$ 
 \begin{align*}
     \varphi(\Phi(f)) = C_\Phi^*(\varphi)(f) = \delta_{g(z)}(f).
 \end{align*}
 As this equality holds for any $\varphi$ in the fiber over $z$ we obtain that $\Phi(f)$ is constant in $\pi_\infty^{-1}(z)$ so by Lemma \ref{farmer} it is continuously extendable to $z$ where it coincides with $f \circ g(z)$. This in turn implies that $\Phi(f)$ is in $\mathcal {H}^\infty(\B)$ and is continuously extendable to $S_{\ell_2}$ where it coincides with $f \circ g$ so we must have that $\Phi(f) = C_g(f)$.
 \end{proof}
 
 Through the previous result we see that any function $g \in \mathcal {H}^\infty(\B,\ell_2)$ with $\|g\|= 1$ and $g(\B) \subset \B$ that can be continuously extendable to $S_{\ell_2}$ and such that $g(S_{\ell_2}) \subset S_{\ell_2}$ (as we mentioned, $Id$ is an example of this) has a singleton middle fiber. Further, any linear isometry $T:\ell_2 \to \ell_2$ (not necessarily onto) when restricted to the unit ball satisfies the above conditions (e.g. the shift operator $S(x_1,x_2,x_3,\ldots) = (0,x_1,x_2,x_3,\ldots)$) and thus has a singleton fiber.
 
 A, perhaps more general, restatement of the previous lemma can be obtained by noting that the condition $g(S_{\ell_2}) \subset S_{\ell_2}$ is used to ensure that whenever $\Phi$ lies in the fiber over $g$, the corresponding mapping $C_\Phi^*$ maps any $\varphi$ in $\pi_\infty^{-1}(z)$ to the same homomorphism in $\M_u(\B)$ whenever $z \in S_{\ell_2}$.

\begin{lemma}
\label{general}
Let $\Phi \in \M_{u,\infty}(\B,\B)$ such that there exists a continuous function $h: S_{\ell_2} \to \overline{B}_{\ell_2}$ satisfying $C_{\Phi}^*(\pi_{\infty}^{-1}(z)) = \{\delta_{h(z)}\}$ for all $z \in S_{\ell_2}$. Then $\Phi$ coincides with the homomorphism $C_{\xi(\Phi)}$.
\end{lemma}

We do not know, however, if there are homomorphisms $\Phi$ in $\M_{u,\infty}(\B,\B)$ satisfying the conditions of Lemma \ref{general} which are not in fibers of  functions fulfilling the conditions of Lemma \ref{Singleton middle fiber}.


So far we have shown middle fibers consisting only of the corresponding composition homomorphism. We can say that they behave like edge fibers. As we have anticipated, there are other middle fibers containing a ball (so, acting as interior fibers). Again we begin with a prototype example.

\begin{example}
\label{ge1}
Let $g \in \mathcal {H}^\infty(\B,\ell_2)$, given by $g(x) = x_1 e_1$. Then $g$ satisfies $\|g\| = 1$, $g(\B) \subset \B$ and there exists an analytic injection $\Psi_g:B_{\mathcal {H}^\infty(\B,\ell_2)} \to \mathscr{F}(g)$.

For every $h \in B_{\mathcal {H}^\infty(\B,\ell_2)}$, $k \in \N$ let $s_{h,k} \in \mathcal {H}^\infty(\B,\ell_2)$ be defined by
\begin{align*}
    s_{h,k}(x) = x_1e_1 + x_2\left(\sum_{j=1}^\infty\langle h(x),e_j\rangle  e_{k+j}\right), \quad \textrm{for all }x\in \B.
\end{align*}
Note that 
\begin{align}
\label{ineqge1}
    \left\|s_{h,k}(x)\right\|^2 &= |x_1|^2 + |x_2|^2\|h(x)\|^2 \leq |x_1|^2 + |x_2|^2 < 1.
\end{align}
Thus we see that $s_{h,k}(\B) \subset \B$. We now define for each $k \in \N$ the composition homomorphism $\Psi_k: B_{\mathcal {H}^\infty(\B,\ell_2)} \to \M_{u,\infty}(\B,\B)$ as $\Psi_k(h)(f) = C_{s_{h,k}}(f) = f \circ s_{h,k}$.

Take $\Psi_g$ to be a weak-star limit along a ultrafilter of the sequence $(\Psi_k)_k$. Now, proceeding as in Theorem \ref{inyVectorial} we can see that $\Psi_g$ is an analytic injection whose image is contained in $\mathscr{F}(g)$.



\end{example}

The arguments in the previous example can be slightly modified to reach more cases as we see in the following theorem. The hypothesis go in line with the idea that we need conditions that ensure we have some `room' to work with to obtain an analytic injection. This is done by ensuring  (perhaps after a change of coordinates) that $g$ does not depend on at least one variable. Additionally, as we want to extend these ideas  beyond functions $g$ with $g(0) = 0$, we include the condition $\|g(0)\| + \|g - g(0)\|\le 1$ and make use of the tools from Example $\ref{ge1}$ for $g - g(0)$.

\begin{theorem}
\label{middleTheorem}
Let $g \in S_{\mathcal {H}^\infty(\B,\ell_2)}$ with $g(\B) \subset \B$. If $g$ satisfies:
\begin{enumerate}
    \item $\|g(0)\| + \|g - g(0)\| \leq 1.$
    \item There is a linear transformation $P: \ell_2 \to \ell_2$ with $\|P\| \leq 1$ and non-trivial kernel such that $g = g \circ P$.
\end{enumerate}
Then there exists an analytic injection $\Psi_g:B_{\mathcal {H}^\infty(\B,\ell_2)} \to \mathscr{F}(g).$
\end{theorem}

\begin{proof}
We first prove the result with the additional hypothesis $g(0) = 0$. Fix $w \in S_{\ell_2}\cap Ker(P)$. We define for each $k \in \N, h \in B_{\mathcal {H}^\infty(\B,\ell_2)}$ the function $s_{h,k} \in \mathcal {H}^\infty(\B,\ell_2)$ as

\begin{align*}
   s_{h,k} (x) = \sum_{j=1}^k \langle g(x), e_j \rangle e_j + \langle x , w \rangle \left(\sum_{j=1}^\infty\langle h(x),e_j\rangle  e_{k+j}\right).
\end{align*}

Note that for every $x \in \B$ we have that
\begin{equation}
    \label{ineq1}
    \|s_{h,k} (x) \|^2 \leq \sum_{j=1}^k |\langle g(x), e_j \rangle|^2 + |\langle x , w \rangle|^2
    \leq \|g(x)\|^2 +  |\langle x , w \rangle|^2.\\
\end{equation}
Now, since $g(0) = 0$ and $\|g\| \leq 1$ we can apply Schwarz Lemma to $g = g \circ P$ and obtain
\begin{align}
\label{ineq2}
\begin{split}    
   \|g(x)\|^2 +  |\langle x , w \rangle|^2 \leq \|P(x)\|^2 + |\langle x , w \rangle|^2 &= \|P( x - \langle x , w\rangle w)\|^2 + |\langle x , w \rangle|^2\\
    &\leq \|x - \langle x , w\rangle w\|^2 + |\langle x , w \rangle|^2 < 1.
    \end{split}
\end{align}
Joining both inequalities we obtain that $ s_{h,k} (\B) \subset \B$. For $k \in \N$ let $\Psi_k: B_{\mathcal {H}^\infty(\B,\ell_2)} \to \M_{u,\infty}(\B,\B)$ be given by $\Psi_k(h)(f) = C_{s_{h,k}}(f) = f \circ s_{h,k}$.
The proof of this case finishes as usual, defining $\Psi_g$ to be a weak-star limit along a ultrafilter of the sequence $(\Psi_k)_k$ and repeating the arguments used in Theorem \ref{inyVectorial}.

For the general case, let $\widetilde{g}= g - g(0)$ and set for every $k \in \N$ the function $s_{h,k} \in \mathcal {H}^\infty(\B,\ell_2)$ as

\begin{equation*}
    s_{h,k}(x) = \sum_{j=1}^k \langle  g(x), e_j \rangle e_j + \|\widetilde{g}\| \langle x , w \rangle \left(\sum_{j=1}^\infty\langle h(x),e_j\rangle  e_{k+j}\right).
\end{equation*}
 Repeating the computation done in \eqref{ineq1} and \eqref{ineq2} for 
\[ \sum_{j=1}^k \langle \widetilde g(x), e_j \rangle e_j + \|\widetilde{g}\| \langle x , w \rangle \left(\sum_{j=1}^\infty\langle h(x),e_j\rangle  e_{k+j}\right)\] yields that 
\begin{align*}
    \left\|\sum_{j=1}^k \langle \widetilde g(x), e_j \rangle e_j + \|\widetilde{g}\| \langle x , w \rangle \left(\sum_{j=1}^\infty\langle h(x),e_j\rangle  e_{k+j}\right)\right\|^2 < \|\widetilde g\|^2,
\end{align*}
for all $x \in \B$. Now the hypothesis $\|g(0)\| + \|g - g(0)\| \leq 1$ ensures that each $s_{h,k}(\B) \subset (\B)$. The argument ends by proceeding as in the first case.
\end{proof}

\begin{example}
Fix $m \in \N$, $Q:\ell_2 \to \ell_2$ an $m-$homogeneous polynomial with $\|Q\|_{\B}=1 $ and a linear projection $P:\ell_2 \to \ell_2$. Define for $z \in S_{\ell_2}$ and $0 < \alpha < 1$ the function $g: \B \to \ell_2$ by

\begin{align*}
    g(x) = \alpha ( Q \circ P) (x)  + (1 - \alpha)z.
\end{align*}
Note that $g(0) = (1 - \alpha)z$ so for every $x \in \B$ we have the inequality
\begin{align*}
    \|g(0)\| + \|g(x) - g(0)\| = (1-\alpha)\|z\| + \alpha \|Q \circ P(x)\| < 1.
\end{align*}
Further, since $P$ satisfies $P \circ P = P$, it follows that $g = g \circ P$. As a result $g$ satisfies the conditions of Theorem~\ref{middleTheorem} and thus there is an analytic copy of a ball in the fiber $\mathscr{F}(g)$.
\end{example} 

Throughout this section we have presented conditions for middle fibers to be either isolated or to contain a ball. However, as we do not know the complete description of the spectrum, it remains open if these are the only possible alternatives for the remaining fibers. For example we do not have a clue of what happen with the (middle) fiber over the function $g\in\mathcal {H}^\infty(\B,\ell_2)$ given by $g(x)=(x_n^2)_{n\in\mathbb N}$.
This function is clearly extended to $S_{\ell_2}$ but the extension do not map $S_{\ell_2}$ into $S_{\ell_2}$. Also, since $g(x)=0$ if and only if $x=0$ it can not exist a projection $P$ with non trivial kernel such that $g\circ P=g$. Thus, $g$ does not satisfy the conditions of Lemma \ref{Singleton middle fiber} nor  those of Theorem \ref{middleTheorem}. There are, of course, many functions not satisfying the conditions of the referred Lemma and Theorem but this mapping $g$ is a very natural element of $\mathcal {H}^\infty(\B, \ell_2)$ so it might be interesting to have a description of its fiber.

\section{Vector-Valued Cluster Sets}

Before focusing on the specific case of $\ell_2$ we present some general results regarding vector-valued cluster sets for $X, Y$ Banach spaces.

Let $g\in \overline{B}_{\mathcal {H}^\infty(B_Y,X^{**})}$, $f \in \A_u(B_X)$. We define the vector-valued cluster set of $f$ over $g$ as 

\begin{align*}
    \mathcal{C\ell}_{\M_{u,\infty}(B_X,B_Y)}(f,g) = \{ h \in \mathcal {H}^\infty(B_Y) \colon \exists \ (g_\alpha) \subset B_{\mathcal {H}^\infty(B_Y,X^{**})}, g_\alpha \xrightarrow{w^*}g \\\text{ satisfying } C_{g_\alpha}(f)(y) \to h(y) \ \forall y \in B_Y\}.
\end{align*}
Here $ g_\alpha \xrightarrow{w^*}g$ denotes the weak-star convergence associated to the identification
\begin{equation} \label{H infinito es dual}
    \mathcal {H}^\infty(B_Y,X^{**}) = \mathcal{L}(\mathcal{G}^\infty(B_Y),X^{**})  = (\mathcal{G}^\infty(B_Y) \widehat\otimes_\pi  X^{*})^*.
\end{equation}
In this topology, $g_\alpha \xrightarrow{w^
*}g$ whenever $g_\alpha(y)(x^*) \to g(y)(x^{*})$ for all $y \in B_Y$, $x^{*} \in X^{*}$. We direct the reader to \cite{Mujica} for the proof of the first identification and further reference.
Whenever there is no ambiguity regarding the vector-valued spectrum referenced we simply write  $\mathcal{C\ell}(f,g)$ to denote the cluster set of $f$ over $g$.

We mentioned in the introduction that the scalar-valued cluster set was first considered by I.J. Schark in \cite{Schark} for the unit disk $\D$. There it is shown that the cluster set $\mathcal{C\ell}_{\D}(f,z_0)$ is a compact subset of $\widehat{f}(\pi^{-1}(z_0))$, where $\widehat{f}(\varphi) = \varphi(f).$ The inclusion 
\begin{align}
    \label{ClusterProblem}
    \mathcal{C\ell}_{B_X}(f,z) \subset \widehat{f}(\pi^{-1}(z))
\end{align}
still holds for a general Banach space (see for instance \cite[Lem. 2.2]{AronCarandoGamelinLassalleMaestre}). Whether or not the equality in \eqref{ClusterProblem} holds is known as the \textsl{Cluster Problem} and remains open. There are, however, positive results for instance in \cite{JohnsonOrtegaCastillo} and \cite{AronCarandoGamelinLassalleMaestre} where the equality in \eqref{ClusterProblem} is proved for every $f \in \A_u(\B)$ and every $z \in \overline B_{\ell_2}$.

The vector-valued cluster set $\mathcal{C\ell}_{\M_{u,\infty}(B_X,B_Y)}(f,g)$ is analogously related to the evaluations $\Phi(f) = \widehat{f}(\Phi)$ for $\Phi \in \mathscr{F}(g)$ as we show in the following Lemma:

\begin{lemma}
\label{inclusion}
Let $g\in \overline{B}_{\mathcal {H}^\infty(B_Y,X^{**})}$, $f \in \A_u(B_X)$. Then the vector-valued cluster set $\mathcal{C\ell}(f,g)$ is a weak-star compact subset of $\widehat{f}(\mathscr{F}(g))$.
\end{lemma}

\begin{proof}
Let $h \in \mathcal{C\ell}(f,g)$. Then there exists a net $g_\alpha$ in $B_{\mathcal {H}^\infty(B_Y,X^{**})}$ such that $g_\alpha \xrightarrow{w*}g$ and $C_{g_\alpha}(f)(y)\to h(y)$ for all $y \in B_Y$. Since $(C_{g_\alpha})_\alpha$ is a subset of the compact set $\M_{u,\infty}(B_X,B_Y)$ we can find a converging sub-net $(C_{g_{\alpha_\beta}})_\beta$ with limit $\Phi.$ Then $\Phi(f)(y) = \lim_\beta C_{g_{\alpha_\beta}}(f)(y) = h(y)$.

Additionally, since $g_\alpha \xrightarrow{w^*} g$ means that $g_\alpha(y)(x^*) \to g(y)(x^*)$ for all $y \in B_Y, x^* \in X^*$, we have
\begin{align*}
    \xi(\Phi)(y)(x^*) = \Phi(x^*)(y) = \lim_\beta C_{g_{\alpha_\beta}}(x^*)(y) = \lim_\alpha g_{\alpha_\beta}(y)(x^*) = g(y)(x^*).
\end{align*}
 This proves the inclusion $\mathcal{C\ell}(f,g) \subset \widehat{f}(\mathscr{F}(g))$.
 
Now we move on to show that the cluster $\mathcal{C\ell}(f,g)$ is weak-star compact. It is readily seen that $\mathcal{C\ell}(f,g) \subset \|f\|\overline{B}_{\mathcal {H}^\infty(B_Y)}$ so it only remains to show that $\mathcal{C\ell}(f,g)$ is  weak-star closed. 

Let $\mathcal{U}$ denote the set of all weak-star neighborhoods $U$ of $g$ in $\mathcal {H}^\infty(B_Y,X^{**})$. To simplify the notation we put $\widetilde U =U\cap B_{\mathcal {H}^\infty(B_Y,X^{**})}$ for each $U\in\mathcal U$. 

We are going to prove the equality
\begin{align*}
    \mathcal{C\ell}(f,g) = \bigcap\limits_{U \in \mathcal{U}} \overline{\widehat{f}(j(\widetilde U)) }^{w*},
\end{align*}
where $j:\mathcal {H}^\infty(B_Y,X^{**})\to \M_{u,\infty}(B_x,B_Y)$ is the mapping defined by $j(u)=C_u$.

Let $h \in \bigcap\limits_{U \in \mathcal{U}} \overline{\widehat{f}(j(\widetilde U)) }^{w*}$. Then for every $U \in \mathcal{U}$ and any weak-star neighborhood $V$ of $h$ there exists a function $g_{U,V}$ such that
\begin{enumerate}
    \item $g_{U,V} \in U \cap B_{\mathcal {H}^\infty(B_Y)}$.
    \item $C_{g_{U,V}}(f) \in V$.
\end{enumerate}.

We consider all possible pairs $(U,V)$ with the order given by the reverse inclusion (that is, $( U' ,  V') \geq (U,V)$ whenever both $ U' \subset U$ and $ V' \subset V$). As a result, we have obtained a net $(g_{U,V})_{(U,V)}$ indexed in the pairs of weak-star neighborhoods of $g$ and $h$ respectively. It is clear from construction that this net satisfies
\begin{enumerate}
    \item $g_{U,V} \xrightarrow{w^*}g$.
    \item $C_{g_{U,V}}(f) \xrightarrow{w^*}h$.
\end{enumerate}
We conclude that $h \in \mathcal{C\ell}(f,g)$ and thus the inclusion $\bigcap\limits_{U \in \mathcal{U}} \overline{\widehat{f}(j(\widetilde U)) }^{w*} \subset \mathcal{C\ell}(f,g)$ holds.
To check the converse, let $h \in \mathcal{C\ell}(f,g)$ 
 satisfying 
\begin{enumerate}
    \item $g_\alpha \xrightarrow{w^*}g$.
    \item $C_{g_\alpha}(f) \xrightarrow{w^*}h$.
\end{enumerate}
Given $U\in\mathcal U$, let
 $\alpha_0$ such that  $g_\alpha \in U$ for all $\alpha \geq \alpha_0$. Since $h \in \overline{(C_{g_\alpha)}(f)}^{w*}_{\alpha \geq \alpha_0}$ it follows that $h \in \overline{\widehat{f}(j(\widetilde U))}^{w*}$. Being the neighborhood $U$  arbitrary  the desired equality is proved.
\end{proof}

\subsection{Cluster sets for $\M_{u,\infty}(\B,\B)$}

Back in Section \ref{FiberSection} we presented results regarding the size of the fibers. \textit{A priori}, this is directly related to the size of the sets $\widehat{f}(\mathscr{F}(g))$. However, we show that the tools developed for Section \ref{FiberSection} can be adapted for a set of functions to gain insight on the size of cluster side of the inclusion 
\begin{align*}
    \mathcal{C\ell}(f,g) \subset \widehat{f}(\mathscr{F}(g)).
\end{align*}
This also goes in line with the philosophy of \cite{AlvesCarando} of looking for sets of functions with large cluster sets.

The study of cluster sizes for different functions $f$ differs from that of the fibers in the previous section in many aspects, but mainly it is worth remarking that we now lack the possibility to try different ``test functions'' as the function is fixed when we consider a certain cluster. Consequently we can not take full advantage of the previous results for the fibers as the proofs require a set of test functions instead of a single one. Additionally, it is worth noting that the inclusion $\mathcal{C\ell}(f,g) \subset \widehat{f}(\mathscr{F}(g))$ proved in Lemma \ref{inclusion} tells us that we can only look for non-singleton cluster sets for functions $g$ with non-singleton fibers.

Also, if $f$ can be approximated by finite type polynomials (i. e. polynomials that are linear combinations of products of linear functionals) then $\Phi(f)=C_{\xi(\Phi)}(f)$. Thus in this case for any function $g$ the cluster set is singleton: $\mathcal{C\ell}(f,g) = \widehat{f}(\mathscr{F}(g))=\{C_{\xi(\Phi)}(f)\}$.

By \cite[Lem. 6.2]{aron1995weak}, a function $f\in\A_u(\B)$ can be approximated by finite type polynomials if and only if it is weakly  continuous at every $x\in\B$. 

Hence, the cluster set $\mathcal{C\ell}(f,g)$ could be \textsl{large} whether $f$ is not weakly  continuous at some point and $g$ is an inner or middle function with non-singleton fiber.

We begin by showing that if $g$ is an inner function and $f$ is not weakly  continuous at some point of the image of $g$ then the cluster set ${C\ell}(f,g)$ contains an analytic copy of the complex disk. The proof is inspired by \cite[Th. 3.3]{DimantSinger} (see also \cite[Th. 3.1]{AronFalcoGarciaMaestre}).

\begin{proposition}
Let $g \in B_{\mathcal {H}^\infty(\B,\ell_2)}$ and let $f_0\in\A_u(\B)$ which is not weakly  continuous at $g(x_0)$ for certain $x_0\in\B$. Then, the cluster set $ {C\ell}(f_0,g)$ contains an analytic copy of the complex disk $\D$.
     
\end{proposition}

\begin{proof}
Let $\|g\|<r<1$. If $f_0$ is not weakly  continuous at $g(x_0)$ it should exist a net $(y_\alpha)_{\alpha}\subset \B$ such that $y_\alpha\overset{w^*}{\to} y_0=g(x_0)$ and $|f_0(y_\alpha)-f_0(y_0)|>\varepsilon$, for certain $\varepsilon>0$. Note that since the argument in \cite[Lem. 6.2]{aron1995weak} can be repeated in any ball we can chose that the net $(y_\alpha)_{\alpha}$ is contained in the ball $B_{1-r}(y_0)$. This is important for the well definition (for each $k\in\N$
) of the following mapping:

\begin{align*}
    \Psi_k:\D &\to \M_{u,\infty}(\B,\B)\\
    \Psi_k(\lambda)(f)(x) &= f\left(g(x) + \lambda (y_\alpha -y_0)\right).
\end{align*}

The inclusion $\M_{u,\infty}(\B,\B)\subset \mathcal L(\A_u(\B), \mathcal {H}^\infty(\B))=(\A_u(\B) \widehat{\otimes}_\pi \mathcal G^\infty(\B))^*$ and the fact that each $\Psi_k$ is a bounded analytic mapping, allow us to see   the sequence $(\Psi_k)_k$  contained in the ball of $\mathcal {H}^\infty(\D, \mathcal L(\A_u(\B), \mathcal {H}^\infty(\B))) $ of radius $\|f\|$.
Recall from \eqref{H infinito es dual} that $\mathcal {H}^\infty(\D, \mathcal L(\A_u(\B), \mathcal {H}^\infty(\B)))$ is a dual space. Thus, the weak-star compactness of the ball assures the existence of a net $(\Psi_{k_\alpha})_\alpha$ weak-star convergent to an element $\Psi_g\in \mathcal {H}^\infty(\D, \mathcal L(\A_u(\B), \mathcal {H}^\infty(\B)))$. 

It is easy to see that for any $\lambda\in\D$, $\Psi_g(\lambda)$ is multiplicative. Hence, $\Psi_g\in \mathcal {H}^\infty(\D, \M_{u,\infty}(\B,\B))$. 

Also, note that each $\Psi_k(\lambda)$ is a composition homomorphism. Thus, we have, for any $\lambda\in\D$,  $x\in\B$,
$$\Psi_g(\lambda)(f_0)(x)= \underset{\alpha}{\lim}\Psi_{k_\alpha}(\lambda)(f_0)(x)= \underset{\alpha}{\lim}C_{\xi(\Psi_{k_\alpha}(\lambda))}(f_0)(x).$$

Since $\xi(\Psi_{k_\alpha}(\lambda))$ is weak star convergent to $g$ we obtain that $\Psi_g(\lambda)(f_0) \in {C\ell}(f_0,g)$.


For the injectivity, first, observe that $\Psi_g$ can be continuously extended to $\overline{\D}$. Second, we know

$$
\Psi_g(0)(f_0)(x_0)=f_0(g(x_0))= f_0(y_0) \quad\textrm{and} \quad \Psi_g(1)(f_0)(x_0)=\underset{\alpha}{\lim} f_0(y_\alpha).
$$

Therefore,  $|\Psi_g(0)(f_0)(x_0)-\Psi_g(1)(f_0)(x_0|\ge\varepsilon$ and so  the mapping $\Psi_g(\cdot) (f_0)(x_0)$ belongs to the disk algebra $\A(\D)$ and it is not constant. Hence, it should exist $\lambda_0\in\D$ and $r_0>0$ such that it is injective in the disk $\D(\lambda_0, r_0)$.

This implies that the mapping $\Psi_g(\cdot)(f_0):\D(\lambda_0, r_0)\to {C\ell}(f_0,g)$ is injective. Finally, composing with the mapping from $\D$ to $\D(\lambda_0, r_0)$ given by $[\lambda\mapsto r_0\lambda + \lambda_0]$ we obtain that ${C\ell}(f_0,g)$ contains an analytic copy of $\D$.
\end{proof}

Now, if we consider very special functions $f$ we can prove that  the cluster set ${C\ell}(f,g)$ contains an infinite dimensional ball for any inner function $g \in B_{\mathcal {H}^\infty(\B, \ell_2)}$.

\begin{proposition}
\label{cluster1}
    Let $f_0(x) = \sum_{j=1}^\infty x_j^N$ with $N\in\N_{\geq 2}$. Then for every $g \in B_{\mathcal {H}^\infty(\B,\ell_2)}$ the cluster ${C\ell}(f_0,g)$ contains an analytic copy of the ball  $B_{\mathcal {H}^\infty(\B)}$.
\end{proposition}

\begin{proof}
Fix $g \in B_{\mathcal {H}^\infty(\B,\ell_2)}$, $\|g\|<r<1$. We define for each $k \in \N$ the mapping $\Psi_k: B_{\mathcal {H}^\infty(\B)} \to \M_{u,\infty}(\B,\B)$ given by 
\begin{align*}
    \Psi_k(h)(f)(x) = f\left(\sum_{j=1}^k \langle g(x), e_j \rangle e_j + (1-r)h(x)e_{k+1}\right),
\end{align*}
for $h\in B_{\mathcal {H}^\infty(\B)},\  f\in\A_u(\B),\ x\in\B$.
Clearly, $\Psi_k$ is a bounded analytic mapping and for any $h$, $\Psi_k(h)$ is a composition homomorphism. 



Argumenting as in the previous proposition we know that there is a net $(\Psi_{k_\alpha})_\alpha$ weak-star convergent to an element $\Psi_g\in \mathcal {H}^\infty(B_{\mathcal {H}^\infty(\B)}, \mathcal L(\A_u(\B), \mathcal {H}^\infty(\B)))$.

By definition 
$$\Psi_g(h)(f)(x)= \underset{\alpha}{\lim}\Psi_{k_\alpha}(h)(f)(x)= \underset{\alpha}{\lim}C_{\xi(\Psi_{k_\alpha}(h))}(f)(x).$$

Since $\xi(\Psi_{k_\alpha}(h))$ is weak star convergent to $g$ we obtain that $\Psi_g(h)(f_0) \in {C\ell}(f_0,g)$.
Further, for every $k \in \N$, $h \in B_{\mathcal {H}^\infty(\B)}$ we have 
\begin{align*}
    \Psi_k(h)(f)(x) &= \sum_{j=1}^k \langle g(x), e_j \rangle^N + (1-r)^N h(x)^N.
    \intertext{Taking the limit as $k \to \infty$ we obtain}
    \Psi_g(h)(f)(x) &= \sum_{j=1}^\infty \langle g(x), e_j \rangle^N + (1-r)^N h(x)^N.
\end{align*}
As a result, if $\Psi_g(h)(f) = \Psi_g(h')(f)$ we obtain $h(x)^N = h'(x)^N$ for all $x \in \B$, so by the identity principle there exists $\zeta \in \C$ satisfying $\zeta^N = 1$ such that $h' = \zeta h$. Finally note that the equality $\|h - \zeta h\| = |1 - \zeta|\|h\|$ shows that if we put $\tau=\min\{|1-\zeta|: \zeta^N=1, \zeta\not=1\}$ then for any fixed $h \in B_{\mathcal {H}^\infty(\B)}$ satisfying $\|h\| = \frac12$ we have that $\Psi_g$ is injective when restricted to $B_s(h)$ for $s=\frac{\tau}{2(\tau +1)}$.
In other words, we have seen that the cluster ${C\ell}(f_0,g)$ contains an analytic copy of a ball  $B_s(h)\subset\mathcal {H}^\infty(\B)$.

Finally, composing with the canonical mapping $\chi:B_{\mathcal {H}^\infty(\B)}\to B_s(h)$ given by $\chi(u)=su+h$ we conclude that $\Psi_g\circ\chi: B_{\mathcal {H}^\infty(\B)}\to {C\ell}(f_0,g)$ is an analytic injection.
\end{proof}

The same result can be obtained for functions $g$ satisfying the hypothesis of Theorem~\ref{middleTheorem}.

\begin{proposition}
     Let $f_0(x) =\sum_{j=1}^\infty x_j^N$ with $N\in\N_{\geq 2}$. Then for every $g \in S_{\mathcal {H}^\infty(\B,\ell_2)}$ with $g(\B) \subset \B$ such that
     \begin{enumerate}
    \item $\|g(0)\| + \|g - g(0)\| \leq 1.$
    \item There exists a linear transformation $P: \ell_2 \to \ell_2$ with $\|P\| \leq 1$ and non-trivial kernel such that $g = g \circ P$,
\end{enumerate}
the cluster ${C\ell}(f_0,g)$ contains an analytic copy of the ball  $B_{\mathcal {H}^\infty(\B)}$.
\end{proposition}

\begin{proof}
 Let $\widetilde g = g - g(0)$ and define for each $k \in \N$ the mapping $\Psi_k: B_{\mathcal {H}^\infty(\B)} \to \M_{u,\infty}(\B,\B)$ given by 
 
 \begin{align*}
    \Psi_k(h)(f)(x) = f\left(\sum_{j=1}^k \langle g(x), e_j \rangle e_j + \|\widetilde{g}\| \langle x , w \rangle h(x) e_{k+1}\right).
 \end{align*}
 Now, the result is obtained following the steps of the proof of Proposition~\ref{cluster1} combined with some arguments from the proof of Theorem~\ref{middleTheorem}.
\end{proof}

\bibliographystyle{amsplain}
\bibliography{liografia.bib}
\end{document}